\documentclass[11pt,a4paper]{article}

\usepackage{inputenc}
\usepackage{amsmath}
\usepackage{bm}
\usepackage{bbold}
\usepackage{amsthm}
\usepackage{enumerate}

\usepackage{hyperref}

\title{Direct Solutions to Tropical Optimization Problems with Nonlinear Objective Functions and Boundary Constraints\thanks{Mathematical Methods and Optimization Techniques in Engineering: Proc. 1st Intern. Conf. on Optimization Techniques in Engineering (OTENG '13), WSEAS Press, 2013, pp.~86--91.}}

\author{Nikolai Krivulin\thanks{Faculty of Mathematics and Mechanics, Saint Petersburg State University, 28 Universitetsky Ave., Saint Petersburg, 198504, Russia, 
nkk\textless at\textgreater math.spbu.ru.}\thanks{The work was supported in part by the Russian Foundation for Humanities under Grant \#13-02-00338.}
\and
Karel Zimmermann\thanks{Faculty of Mathematics and Physics, Charles University, 25~Malostranske nam., 118 00 Prague 1, Czech Republic, Karel.Zimmermann\textless at\textgreater mff.cuni.cz}
}

\date{}

\newtheorem{theorem}{Theorem}
\newtheorem{lemma}[theorem]{Lemma}

\begin{document}

\maketitle

\begin{abstract}
We examine two multidimensional optimization problems that are formulated in terms of tropical mathematics. The problems are to minimize nonlinear objective functions, which are defined through the multiplicative conjugate vector transposition on vectors of a finite-dimensional semimodule over an idempotent semifield, and subject to boundary constraints. The solution approach is implemented, which involves the derivation of the sharp bounds on the objective functions, followed by determination of vectors that yield the bound. Based on the approach, direct solutions to the problems are obtained in a compact vector form. To illustrate, we apply the results to solving constrained Chebyshev approximation and location problems, and give numerical examples.
\\

\textbf{Key-Words:} idempotent semifield, tropical optimization problem, boundary constraint, chebyshev location problem, chebyshev approximation.
\\

\textbf{MSC (2010):} 65K10, 15A80, 65K05, 90B85, 41A50
\end{abstract}

\section{Introduction}

Since the early publications in 1960s, tropical (idempotent) mathematics, as the mathematics of idempotent semirings, has found many applications in optimization, control, decision making, and other fields. Over these decades, the theory and practice of idempotent mathematics has been developed in many publications, including \cite{Cuninghamegreen1979Minimax,Zimmermann1981Linear,Baccelli1993Synchronization,Deschutter1996Maxalgebraic,Kolokoltsov1997Idempotent,Golan2003Semirings,Heidergott2006Maxplus,Akian2007Maxplus,Litvinov2007Themaslov,Gondran2008Graphs,Butkovic2010Maxlinear}.

In the literature, there is a range of real-world optimization problems that can be formulated and solved in the tropical mathematics setting to minimize linear and nonlinear objective functions defined on finite-dimensional semimodules over idempotent semifields. Well-known examples include multidimensional problems that arise in job scheduling \cite{Cuninghamegreen1976Projections,Cuninghamegreen1979Minimax,Zimmermann1981Linear,Zimmermann1984Some,Zimmermann2003Disjunctive,Zimmermann2006Interval,Butkovic2010Maxlinear} and location analysis \cite{Zimmermann1984Onmaxseparable,Cuninghamegreen1994Minimax,Krivulin2011Analgebraic,Krivulin2012Anew}.

Many available solution techniques apply iterative computational schemes and provide only particular solutions for the problems \cite{Zimmermann1984Some,Zimmermann1984Onmaxseparable,Zimmermann2003Disjunctive,Zimmermann2006Interval,Butkovic2010Maxlinear}. These techniques are based on numerical algorithms, which produce a solution if any solution exists, or indicate that there is no solution. Other approaches offer direct explicit solutions and, in some cases, can give complete solutions \cite{Cuninghamegreen1976Projections,Cuninghamegreen1979Minimax,Zimmermann1981Linear,Krivulin2011Analgebraic,Krivulin2012Anew}.

In this paper, we consider multidimensional tropical optimization problems with nonlinear objective functions defined through the multiplicative conjugate vector transposition, and with boundary constraints. As the starting point, we take the problem with two-sided boundary constraints, which was first examined and solved with a numerical algorithm in \cite{Zimmermann1984Some}. We consider two particular cases of the problem and obtain direct solutions in a compact vector form. For one of the problems, we offer a complete solution.

We follow a solution approach that is based on the application and further development of the technique, which was proposed in \cite{Krivulin2005Onsolution,Krivulin2009Methods,Krivulin2013Solution}. The technique involves the derivation of the sharp bounds on the objective functions in the problems, followed by determination of vectors that yield the bounds.

The rest of the paper is as follows. We give a short concise overview of the notation and preliminary results in Section~\ref{S-NPR}. Furthermore, in Section~\ref{S-OP}, we outline a class of tropical optimization problems of interest. Section~\ref{S-CP} presents the problems to be solved and then gives direct solutions. Finally, application and numerical examples are discussed in Section~\ref{S-AE}.

\section{Notation and Preliminary Results}
\label{S-NPR}

We start with an overview of notation and results of idempotent algebra to provide an appropriate framework for the analysis of tropical optimization problems to be performed below. The overview primarily follows the presentation of the topic in \cite{Krivulin2005Onsolution,Krivulin2009Methods,Krivulin2013Solution}. For both introductory and advanced material one can consult \cite{Cuninghamegreen1979Minimax,Zimmermann1981Linear,Baccelli1993Synchronization,Deschutter1996Maxalgebraic,Kolokoltsov1997Idempotent,Golan2003Semirings,Heidergott2006Maxplus,Akian2007Maxplus,Litvinov2007Themaslov,Gondran2008Graphs,Butkovic2010Maxlinear} as well. 

\subsection{Idempotent Semifield}

Consider an idempotent semifield $\langle\mathbb{X},\mathbb{0},\mathbb{1},\oplus,\otimes\rangle$, where $\mathbb{X}$ is a carrier set that is closed under addition $\oplus$ and multiplication $\otimes$, and contains the zero $\mathbb{0}$ and the identity $\mathbb{1}$. Addition is idempotent, which implies that $x\oplus x=x$ for all $x\in\mathbb{X}$. In the semifield, for each $x\in\mathbb{X}_{+}$, where $\mathbb{X}_{+}=\mathbb{X}\setminus\{\mathbb{0}\}$, there exists an inverse $x^{-1}$ such that $x^{-1}\otimes x=\mathbb{1}$.

For each $x\in\mathbb{X}_{+}$ and any integer $p\geq0$, exponential notation is routinely defined as follows: $x^{0}=\mathbb{1}$, $\mathbb{0}^{p}=\mathbb{0}$, $x^{p}=x^{p-1}\otimes x$, and $x^{-p}=(x^{-1})^{p}$. Moreover, the semifield is assumed algebraically closed (radicable), which means that the integer power is extendable to the case of rational exponents.

In what follows, we suppress the multiplication sign as in conventional algebra and use the exponential notation only in the above mentioned sense. 

There is a partial order, which is induced on the semifield by idempotent addition such that $x\leq y$ if and only if $x\oplus y=y$. The order is assumed extendable to a total order to make the semifield linearly ordered. Below, the relation symbols and the optimization objectives are considered in terms of this order. 

Addition and multiplication are monotone in each argument, which implies that the inequalities $x\leq u$ and $y\leq v$ involve $x\oplus y\leq u\oplus v$ and $xy\leq uv$.

As an illustration of the idempotent semifields under study, we suggest the real semifield
$$
\mathbb{R}_{\max,+}
=
\langle\mathbb{R}\cup\{-\infty\},-\infty,0,\max,+\rangle.
$$

This semifield is used later to provide application examples of tropical optimization problems.

\subsection{Matrix and Vector Algebra}

We consider matrices with entries from $\mathbb{X}$ and denote the set of matrices with $m$ rows and $n$ columns by $\mathbb{X}^{m\times n}$. For conforming any matrices $\bm{A}$, $\bm{B}$, $\bm{C}$, and scalar $x$, the sum $\bm{A}\oplus\bm{B}$ and the products $\bm{A}\bm{C}$ and $x\bm{A}$ are calculated by the usual rules with the scalar operations $\oplus$ and $\otimes$ in place of ordinary addition and multiplication. Clearly, these matrix operations are component-wise monotone in each argument.

A matrix is called row (column) regular, if it has no rows (columns) with all entries equal to $\mathbb{0}$. A matrix is regular, if it is both row and column regular.

A square matrix that has $\mathbb{1}$ on the diagonal and $\mathbb{0}$ elsewhere is the identity matrix, which is denoted $\bm{I}$.

Let $\mathbb{X}^{n}$ be the set of column vectors over $\mathbb{X}$ with $n$ elements. Vector addition and scalar multiplication are defined component-wise in terms of the scalar operations $\oplus$ and $\otimes$. Both vector operations are component-wise monotone in each argument.

A vector that consists entirely of $\mathbb{0}$ is the zero vector. A vector is regular, if it has no zero elements.

For any nonzero column vector $\bm{x}=(x_{i})$, the multiplicative conjugate transpose is a row vector $\bm{x}^{-}=(x_{i}^{-})$ with components $x_{i}^{-}=x_{i}^{-1}$ if $x_{i}\ne\mathbb{0}$, and $x_{i}^{-}=\mathbb{0}$ otherwise.

If both vectors $\bm{x}$ and $\bm{y}$ are regular, then the component-wise inequality $\bm{x}\leq\bm{y}$ implies the inequality $\bm{x}^{-}\geq\bm{y}^{-}$ and vice versa.

For any two regular vectors $\bm{x},\bm{y}\in\mathbb{X}^{n}$, we define the distance function
\begin{equation}
\rho(\bm{x},\bm{y})
=
\bm{y}^{-}\bm{x}\oplus\bm{x}^{-}\bm{y}.
\label{E-rhoxy}
\end{equation}

Note that, in terms of the semifield $\mathbb{R}_{\max,+}$, the function can be represented in the form
$$
\rho(\bm{x},\bm{y})
=
\max_{1\leq i\leq n}|y_{i}-x_{i}|,
$$
and thus coincides with the Chebyshev norm.

Finally, note that any nonzero column vector $\bm{x}$ satisfies the equality $\bm{x}^{-}\bm{x}=\mathbb{1}$. Moreover, it is not difficult to verify that the matrix inequality $\bm{x}\bm{x}^{-}\geq\bm{I}$ is valid for any regular column vector $\bm{x}$.

We use these facts to solve the following problem: given a matrix $\bm{A}\in\mathbb{X}^{m\times n}$ and a vector $\bm{p}\in\mathbb{X}^{m}$, find all regular vectors $\bm{x}\in\mathbb{X}^{n}$ to satisfy the inequality
\begin{equation}
\bm{A}\bm{x}
\leq
\bm{p}.
\label{I-Axp}
\end{equation}

\begin{lemma}
\label{L-xpA}
For any column-regular matrix $\bm{A}$ and regular vector $\bm{p}$, all regular solutions to \eqref{I-Axp} are given by
\begin{equation}
\bm{x}
\leq
(\bm{p}^{-}\bm{A})^{-}.
\label{I-xpA}
\end{equation}
\end{lemma}
\begin{proof}
Let us verify that both inequalities \eqref{I-Axp} and \eqref{I-xpA} are equivalent. First, we multiply inequality \eqref{I-Axp} on the left by $(\bm{p}^{-}\bm{A})^{-}\bm{p}^{-}$, and then write
$$
\bm{x}
\leq
(\bm{p}^{-}\bm{A})^{-}\bm{p}^{-}\bm{A}\bm{x}
\leq
(\bm{p}^{-}\bm{A})^{-}\bm{p}^{-}\bm{p}
=
(\bm{p}^{-}\bm{A})^{-}
$$
to obtain inequality \eqref{I-xpA}. On the other hand, after left multiplication of \eqref{I-xpA} by the matrix $\bm{A}$, we have
$$
\bm{A}\bm{x}
\leq
\bm{A}(\bm{p}^{-}\bm{A})^{-}
\leq
\bm{p}\bm{p}^{-}\bm{A}(\bm{p}^{-}\bm{A})^{-}
=
\bm{p},
$$
which completes the proof.
\end{proof}

\section{Optimization Problems}
\label{S-OP}

We consider the tropical optimization problems with non-linear objective functions and linear constraints, which were apparently first examined in \cite{Cuninghamegreen1976Projections,Cuninghamegreen1979Minimax,Zimmermann1984Some}. The problems appeared in the analysis of the tropical vector equation $\bm{A}\bm{x}=\bm{p}$ and were motivated by real-world problems in job scheduling. Initially represented in somewhat different forms, the problems are written below in a unified way in terms of multiplicative conjugate transposition. 

Given a matrix $\bm{A}\in\mathbb{X}^{m\times n}$ and a vector $\bm{p}\in\mathbb{X}^{m}$, consider the problem of finding vectors $\bm{x}\in\mathbb{X}^{n}$ that
\begin{equation}
\begin{aligned}
&
\text{minimize}
&&
(\bm{A}\bm{x})^{-}\bm{p},
\\
&
\text{subject to}
&&
\bm{A}\bm{x}
\leq
\bm{p}.
\end{aligned}
\label{P-AxpAxp}
\end{equation}

In \cite{Cuninghamegreen1976Projections,Cuninghamegreen1979Minimax}, this problem was formulated to obtain a best underestimating approximation $\bm{A}\bm{x}$ for $\bm{p}$  with respect to the Chebyshev norm. A direct closed-form solution to the problem was derived within the framework of the minimax algebra theory developed there. 

Suppose that $\bm{g},\bm{h}\in\mathbb{X}^{n}$ are given vectors such that $\bm{g}\leq\bm{h}$ are lower and upper boundary constraints imposed on $\bm{x}$. We now consider a problem
\begin{equation}
\begin{aligned}
&
\text{minimize}
&&
\bm{p}^{-}\bm{A}\bm{x}\oplus(\bm{A}\bm{x})^{-}\bm{p},
\\
&
\text{subject to}
&&
\bm{g}
\leq
\bm{x}
\leq
\bm{h},
\end{aligned}
\label{P-pAxAxpgxh}
\end{equation}
which yields a best approximate solution to the equation $\bm{A}\bm{x}=\bm{p}$ under the boundary constraints. This constrained optimization problem was solved in \cite{Zimmermann1984Some} via a finite polynomial threshold-type algorithm.

In the same context of solving linear equations, an unconstrained version of problem \eqref{P-pAxAxpgxh} in the form
\begin{equation*}
\begin{aligned}
&
\text{minimize}
&&
\bm{p}^{-}\bm{A}\bm{x}\oplus(\bm{A}\bm{x})^{-}\bm{p},
\end{aligned}
\label{P-pAxAxp}
\end{equation*}
and problem \eqref{P-AxpAxp} were examined in \cite{Krivulin2005Onsolution,Krivulin2009Methods,Krivulin2013Solution}. A solution approach was proposed, which involves the evaluation of sharp bounds on the objective function. Using this approach, direct solutions to the problems were obtained in a compact vector form.

We now assume that one more vector $\bm{q}\in\mathbb{X}^{m}$ is given and consider the unconstrained problem 
\begin{equation}
\begin{aligned}
&
\text{minimize}
&&
\bm{q}^{-}\bm{A}\bm{x}\oplus(\bm{A}\bm{x})^{-}\bm{p}.
\end{aligned}
\label{P-qAxAxp}
\end{equation}

Setting $\bm{A}=\bm{I}$ gives an unconstrained problem
\begin{equation}
\begin{aligned}
&
\text{minimize}
&&
\bm{q}^{-}\bm{x}\oplus\bm{x}^{-}\bm{p}.
\end{aligned}
\label{P-qxxp}
\end{equation}

Problem \eqref{P-qxxp} together with two constrained problems, one having inequality constraints,
\begin{equation*}
\begin{aligned}
&
\text{minimize}
&&
\bm{q}^{-}\bm{x}\oplus\bm{x}^{-}\bm{p},
\\
&
\text{subject to}
&&
\bm{A}\bm{x}
\leq
\bm{x},
\end{aligned}
\end{equation*}
and the other with equality constraints,
\begin{equation*}
\begin{aligned}
&
\text{minimize}
&&
\bm{q}^{-}\bm{x}\oplus\bm{x}^{-}\bm{p},
\\
&
\text{subject to}
&&
\bm{A}\bm{x}
=
\bm{x},
\end{aligned}
\end{equation*}
appeared in solving multidimensional single facility location problems with the Chebyshev distance.

These problems were investigated in \cite{Krivulin2011Analgebraic,Krivulin2012Anew}, where the application of the above mentioned approach provided exact solutions to the problems. The solution obtained for problem \eqref{P-qxxp} was complete.

Below, we consider extended problems that combine the objective functions at \eqref{P-qAxAxp} and \eqref{P-qxxp} with the left or both boundary constraints at \eqref{P-pAxAxpgxh}. 

\section{Constrained Problems}
\label{S-CP}

We are now in a position to present our main results on the solution to tropical optimization problems with boundary constraints. The exact solutions to be given are based on the use and further development of the techniques offered in \cite{Krivulin2005Onsolution,Krivulin2009Methods,Krivulin2013Solution}.

\subsection{Lower and Upper Boundary Constraints}

We start with a complete solution to the following problem: given vectors $\bm{p},\bm{q},\bm{g},\bm{h}\in\mathbb{X}^{n}$, find regular vectors $\bm{x}\in\mathbb{X}^{n}$ that
\begin{equation}
\begin{aligned}
&
\text{minimize}
&&
\bm{q}^{-}\bm{x}\oplus\bm{x}^{-}\bm{p},
\\
&
\text{subject to}
&&
\bm{g}
\leq
\bm{x}
\leq
\bm{h}.
\end{aligned}
\label{P-qxxpgxh}
\end{equation}

The next result offers a straightforward solution to the problem under fairly general assumptions.  
\begin{theorem}
\label{T-qxxpgxh}
Let $\bm{p}$ and $\bm{q}$ be regular vectors, $\bm{g}$ and $\bm{h}$ be vectors such that $\bm{g}\leq\bm{h}$, and $\Delta=\sqrt{\bm{q}^{-}\bm{p}}$. Denote
\begin{equation}
\mu
=
\Delta\oplus\bm{q}^{-}\bm{g}\oplus\bm{h}^{-}\bm{p}.
\label{E-muDeltaqqghp}
\end{equation}

Then the minimum in problem \eqref{P-qxxpgxh} is equal to $\mu$ and attained if and only if
\begin{equation}
\mu^{-1}\bm{p}\oplus\bm{g}
\leq
\bm{x}
\leq
(\mu^{-1}\bm{q}^{-}\oplus\bm{h}^{-})^{-}.
\label{I-mupxmuq}
\end{equation}
\end{theorem}
\begin{proof}
Consider the objective function in the problem and show that $\mu$ is its lower bound. Take an arbitrary regular $\bm{x}$ that satisfies the constraints and examine
$$
r
=
\bm{q}^{-}\bm{x}\oplus\bm{x}^{-}\bm{p}.
$$

From the equality, we have two inequalities
$$
r
\geq
\bm{q}^{-}\bm{x},
\qquad
r
\geq
\bm{x}^{-}\bm{p}.
$$

The first inequality and the left boundary constraint provide a lower bound $r\geq\bm{q}^{-}\bm{x}\geq\bm{q}^{-}\bm{g}$.

Due to Lemma~\ref{L-xpA}, the first inequality is equivalent to the inequality $\bm{x}\leq r\bm{q}$. The substitution into the second inequality gives $r\geq\bm{x}^{-}\bm{p}\geq r^{-1}\bm{q}^{-}\bm{p}$, which yields another lower bound $r\geq\sqrt{\bm{q}^{-}\bm{p}}=\Delta$.

Finally, the second inequality and the right boundary constraint lead to
$r\geq\bm{x}^{-}\bm{p}\geq\bm{h}^{-}\bm{p}$.

By combining the bounds, we obtain
$$
r
\geq
\Delta
\oplus
\bm{q}^{-}\bm{g}
\oplus
\bm{h}^{-}\bm{p}
=
\mu.
$$

To find all solutions to the problem, we examine the equation
$$
\bm{q}^{-}\bm{x}\oplus\bm{x}^{-}\bm{p}
=
\mu.
$$

Since $\mu$ is a lower bound, the equation has the same regular solutions as the inequality
$$
\bm{q}^{-}\bm{x}\oplus\bm{x}^{-}\bm{p}
\leq
\mu,
$$
which is itself equivalent to the pair of inequalities
$$
\bm{q}^{-}\bm{x}
\leq
\mu,
\qquad
\bm{x}^{-}\bm{p}
\leq
\mu.
$$

The application of Lemma~\ref{L-xpA} to the inequalities leads to the solutions
$$
\bm{x}
\leq
\mu\bm{q},
\qquad
\bm{x}
\geq
\mu^{-1}\bm{p}.
$$

By coupling these solutions with the boundary constraints, we arrive at solution \eqref{I-mupxmuq}. 
\end{proof}

It is easy to see from the proof of the theorem that, if the left, right, or both boundaries are not specified in problem \eqref{P-qxxpgxh}, the solutions \eqref{E-muDeltaqqghp} and \eqref{I-mupxmuq} take reduced forms. Specifically, we have
\begin{gather*}
\mu
=
\Delta\oplus\bm{q}^{-}\bm{g},
\\
\mu^{-1}\bm{p}\oplus\bm{g}
\leq
\bm{x}
\leq
\mu\bm{q},
\end{gather*}
for the case when only the constraint $\bm{x}\geq\bm{g}$ is given,
\begin{gather*}
\mu
=
\Delta\oplus\bm{h}^{-}\bm{p},
\\
\mu^{-1}\bm{p}
\leq
\bm{x}
\leq
(\mu^{-1}\bm{q}^{-}\oplus\bm{h}^{-})^{-},
\end{gather*} 
for the constraint $\bm{x}\leq\bm{h}$, and
\begin{gather*}
\mu
=
\Delta,
\\
\mu^{-1}\bm{p}
\leq
\bm{x}
\leq
\mu\bm{q},
\end{gather*} 
if no boundary constraints are imposed.

Note that the last solution coincides with the result for this case, which was obtained in \cite{Krivulin2012Anew}.

\subsection{A One-Sided Boundary Constraint}

Given a matrix $\bm{A}\in\mathbb{X}^{m\times n}$ together with vectors $\bm{p},\bm{q}\in\mathbb{X}^{m}$ and $\bm{g}\in\mathbb{X}^{n}$, consider the problem of finding regular vectors $\bm{x}\in\mathbb{X}^{n}$ that
\begin{equation}
\begin{aligned}
&
\text{minimize}
&&
\bm{q}^{-}\bm{A}\bm{x}\oplus(\bm{A}\bm{x})^{-}\bm{p},
\\
&
\text{subject to}
&&
\bm{x}
\geq
\bm{g},
\end{aligned}
\label{P-qAxAxpxg}
\end{equation}

A direct solution to the problem can be derived using similar arguments as in the previous theorem.
\begin{theorem}
\label{T-qAxAxpxg}
Suppose that $\bm{A}$ is a regular matrix, $\bm{p}$ and $\bm{q}$ are regular vectors, $\bm{g}$ is an arbitrary vector, and $\Delta=\sqrt{(\bm{A}(\bm{q}^{-}\bm{A})^{-})^{-}\bm{p}}$. Denote
\begin{equation*}
\mu
=
\Delta\oplus\bm{q}^{-}\bm{A}\bm{g}.
\label{E-muDeltaqAg}
\end{equation*}

Then the minimum in problem \eqref{P-qAxAxpxg} is equal to $\mu$ and attained at the vector
\begin{equation*}
\bm{x}
=
\mu(\bm{q}^{-}\bm{A})^{-}.
\label{E-xmuqA}
\end{equation*}
\end{theorem}

\begin{proof}
Take a regular $\bm{x}\geq\bm{g}$ and consider the value
$$
r
=
\bm{q}^{-}\bm{A}\bm{x}\oplus(\bm{A}\bm{x})^{-}\bm{p}.
$$

We have two inequalities
$$
r
\geq
\bm{q}^{-}\bm{A}\bm{x},
\qquad
r
\geq
(\bm{A}\bm{x})^{-}\bm{p}.
$$

By combining the first inequality with the constraint, we obtain one bound $r\geq\bm{q}^{-}\bm{A}\bm{x}\geq\bm{q}^{-}\bm{A}\bm{g}$.

Furthermore, we apply Lemma~\ref{L-xpA} to solve the first inequality in the form $\bm{x}\leq r(\bm{q}^{-}\bm{A})^{-}$. The solution taken together with the second inequality give
$r\geq(\bm{A}\bm{x})^{-}\bm{p}\geq r^{-1}(\bm{A}(\bm{q}^{-}\bm{A})^{-})^{-}\bm{p}$, which leads to another bound $r\geq\sqrt{(\bm{A}(\bm{q}^{-}\bm{A})^{-})^{-}\bm{p}}=\Delta$.

Both bounds can be written together as
$$
r
\geq
\Delta
\oplus
\bm{q}^{-}\bm{A}\bm{g}
=
\mu.
$$

We now verify that that the minimum value $\mu$ is attained at the vector $\bm{x}=\mu(\bm{q}^{-}\bm{A})^{-}\geq\bm{g}$. First, we ascertain that
$$
\bm{x}
=
(\Delta\oplus\bm{q}^{-}\bm{A}\bm{g})(\bm{q}^{-}\bm{A})^{-}
\geq
(\bm{q}^{-}\bm{A})^{-}\bm{q}^{-}\bm{A}\bm{g}
\geq
\bm{g}.
$$

Finally, we substitute this vector $\bm{x}$ into the objective function. Considering that $\Delta\leq\mu$, we obtain
\begin{multline*}
\bm{q}^{-}\bm{A}\bm{x}\oplus(\bm{A}\bm{x})^{-}\bm{p}
\\
=
\mu\bm{q}^{-}\bm{A}(\bm{q}^{-}\bm{A})^{-}\oplus\mu^{-1}(\bm{A}(\bm{q}^{-}\bm{A})^{-})^{-}\bm{p}
\\
=
\mu\oplus\mu^{-1}\Delta^{2}
=
\mu.
{\qedhere}
\end{multline*}
\end{proof}

Suppose that no lower bound is defined in the problem, and thus we can put $\bm{g}=\mathbb{0}$. In this case, the solution offered by the theorem becomes the same as that derived in \cite{Krivulin2012Anew}.

\section{Applications and Examples} 
\label{S-AE}

In this section we discuss applications of the results obtained and give numerical examples. Below, we take $\mathbb{R}_{\max,+}$ as the carrier semifield and thus apply the results under the assumption that $\mathbb{X}=\mathbb{R}_{\max,+}$.

\subsection{A Constrained Location Problem}

Consider the following problem, which arises in the solution of single facility location problems in $\mathbb{R}^{n}$ with the Chebyshev norm \cite{Krivulin2011Analgebraic,Krivulin2012Anew}. Given points $\bm{r}$, $\bm{s}$, $\bm{g}$, and $\bm{h}$, locate a new point $\bm{x}$ that minimizes the maximum of the Chebyshev distances from $\bm{x}$ to $\bm{r}$ and to $\bm{s}$, and satisfies the boundary constraints $\bm{g}\leq\bm{x}\leq\bm{h}$.

To solve the location problem, we first apply \eqref{E-rhoxy} to represent the maximum distance as follows:
\begin{multline*}
\rho(\bm{x},\bm{r})\oplus\rho(\bm{x},\bm{s})
=
\bm{r}^{-}\bm{x}\oplus\bm{x}^{-}\bm{r}
\oplus
\bm{s}^{-}\bm{x}\oplus\bm{x}^{-}\bm{s}
\\
=
(\bm{r}^{-}\oplus\bm{s}^{-})\bm{x}\oplus\bm{x}^{-}(\bm{r}\oplus\bm{s}),
\end{multline*}
and then formulate the problem in the form
\begin{equation}
\begin{aligned}
&
\text{minimize}
&&
(\bm{r}^{-}\oplus\bm{s}^{-})\bm{x}\oplus\bm{x}^{-}(\bm{r}\oplus\bm{s}),
\\
&
\text{subject to}
&&
\bm{g}
\leq
\bm{x}
\leq
\bm{h},
\end{aligned}
\label{P-rsxxrsgxh}
\end{equation}

It remains to reduce the problem to \eqref{P-qxxpgxh} by substituting $\bm{p}=\bm{r}\oplus\bm{s}$ and $\bm{q}^{-}=\bm{r}^{-}\oplus\bm{s}^{-}$ and then apply Theorem~\ref{T-qxxpgxh} to obtain a complete direct solution.
\begin{lemma}
Suppose that $\bm{r}$ and $\bm{s}$ are regular vectors, $\bm{g}$ and $\bm{h}$ are vectors such that $\bm{g}\leq\bm{h}$, and $\Delta=\sqrt{(\bm{r}^{-}\oplus\bm{s}^{-})(\bm{r}\oplus\bm{s})}$. Denote
\begin{equation*}
\mu
=
\Delta\oplus(\bm{r}^{-}\oplus\bm{s}^{-})\bm{g}\oplus\bm{h}^{-}(\bm{r}\oplus\bm{s}).
\label{E-mu1}
\end{equation*}

Then the minimum distance in problem \eqref{P-rsxxrsgxh} is equal to $\mu$ and attained if and only if
\begin{equation*}
\mu^{-1}(\bm{r}\oplus\bm{s})\oplus\bm{g}
\leq
\bm{x}
\leq
(\mu^{-1}(\bm{r}^{-}\oplus\bm{s}^{-})\oplus\bm{h}^{-})^{-}.
\end{equation*}
\end{lemma}

We illustrate this result with a location problem with the given points
$$
\bm{r}
=
\left(
\begin{array}{r}
-3
\\
1
\\
1
\end{array}
\right),
\qquad
\bm{s}
=
\left(
\begin{array}{r}
1
\\
3
\\
-2
\end{array}
\right),
$$
and the boundary points
$$
\bm{g}
=
\left(
\begin{array}{c}
0
\\
0
\\
0
\end{array}
\right),
\qquad
\bm{h}
=
\left(
\begin{array}{c}
1
\\
1
\\
1
\end{array}
\right).
$$

First, we find vectors
$$
\bm{r}\oplus\bm{s}
=
\left(
\begin{array}{c}
1
\\
3
\\
1
\end{array}
\right),
\qquad
\bm{r}^{-}\oplus\bm{s}^{-}
=
\left(
\begin{array}{ccc}
3
&
-1
&
2
\end{array}
\right),
$$
and then calculate
$$
\Delta
=
2,
\quad
(\bm{r}^{-}\oplus\bm{s}^{-})\bm{g}
=
3,
\quad
\bm{h}^{-}(\bm{r}\oplus\bm{s})
=
2.
$$

Since $\mu=3$, we finally have the solution in the form
$$
\left(
\begin{array}{r}
0
\\
0
\\
0
\end{array}
\right)
\leq
\bm{x}
\leq
\left(
\begin{array}{r}
0
\\
1
\\
1
\end{array}
\right).
$$

\subsection{A Constrained Approximation Problem}

Let a matrix $\bm{A}$ and vectors $\bm{p}$ and $\bm{g}$ of appropriate size be given over $\mathbb{R}_{\max,+}$. Suppose that one has to determine a best approximation of $\bm{p}$ by $\bm{A}\bm{x}$ in terms of the Chebyshev norm $\rho(\bm{A}\bm{x},\bm{p})$, subject to the boundary constraints $\bm{x}\geq\bm{g}$. This problem has natural interpretations in many areas, including real-world problems in job scheduling (see, e.g., \cite{Cuninghamegreen1976Projections,Cuninghamegreen1979Minimax,Zimmermann1984Some}).

With definition \eqref{E-rhoxy}, we immediately arrive at the problem to find regular vectors $\bm{x}$ that
\begin{equation}
\begin{aligned}
&
\text{minimize}
&&
\bm{p}^{-}\bm{A}\bm{x}\oplus(\bm{A}\bm{x})^{-}\bm{p},
\\
&
\text{subject to}
&&
\bm{x}
\geq
\bm{g}.
\end{aligned}
\label{P-pAxApxg}
\end{equation}

By the substitution $\bm{q}=\bm{p}$ in Theorem~\ref{T-qAxAxpxg}, we get the following solution to the approximation problem.
\begin{lemma}
\label{T-pAxAxpxg}
Suppose that $\bm{A}$ is a regular matrix, $\bm{p}$ is a regular vector, $\bm{g}$ is an arbitrary vector, and $\Delta=\sqrt{(\bm{A}(\bm{p}^{-}\bm{A})^{-})^{-}\bm{p}}$. Denote
\begin{equation*}
\mu
=
\Delta\oplus\bm{p}^{-}\bm{A}\bm{g}.
\label{E-muDeltapAg}
\end{equation*}

Then the least approximation error in problem \eqref{P-pAxApxg} is equal to $\mu$ and attained at the vector
\begin{equation*}
\bm{x}
=
\mu(\bm{p}^{-}\bm{A})^{-}.
\label{E-xmupA}
\end{equation*}
\end{lemma}

We now consider an approximation problem under the assumption that $m=n=3$ and
\begin{gather*}
\bm{A}
=
\left(
\begin{array}{crc}
1 & -1 & 1
\\
3 & 1 & 0
\\
0 & 0 & 2
\end{array}
\right),
\qquad
\bm{p}
=
\left(
\begin{array}{c}
3
\\
4
\\
4
\end{array}
\right),
\\
\bm{g}
=
\left(
\begin{array}{c}
2
\\
2
\\
2
\end{array}
\right).
\end{gather*}

To evaluate the approximation error $\mu$, we first calculate the vectors
$$
(\bm{p}^{-}\bm{A})^{-}
=
\left(
\begin{array}{c}
1
\\
3
\\
2
\end{array}
\right),
\qquad
\bm{A}(\bm{p}^{-}\bm{A})^{-}
=
\left(
\begin{array}{c}
3
\\
4
\\
4
\end{array}
\right).
$$

Furthermore, we successively find
$$
\Delta
=
0
=
\mathbb{1},
\qquad
\bm{A}\bm{g}
=
\left(
\begin{array}{c}
3
\\
5
\\
4
\end{array}
\right),
\qquad
\bm{p}^{-}\bm{A}\bm{g}
=
1,
$$
and then arrive at $\mu=1$.

Finally, we obtain the solution to the approximation problem in the form
$$
\bm{x}
=
\left(
\begin{array}{ccc}
2 & 4 & 3
\end{array}
\right)^{T}.
$$

\bibliographystyle{abbrvurl}

\bibliography{Direct_solutions_to_tropical_optimization_problems_with_nonlinear_objective_functions_and_boundary_constraints}

\end{document}